\documentclass[10pts,english]{amsart} 
\usepackage[utf8]{inputenc}

\usepackage[left=2.5cm,right=2.5cm]{geometry}
\usepackage[english]{babel}
\usepackage{pgf,tikz}
\usetikzlibrary{positioning}
\usetikzlibrary{arrows}
       
\usepackage{mathtools}
\usepackage{amssymb}

\usepackage{amsthm}
\usepackage{amsmath}
 \usepackage{bbm}
\usepackage{algorithm}

\newtheorem{defdeuxieme}{Definition}[section]
\newtheorem{sixieme}{Proposition (Half mass transportation)}[section]
\newtheorem{deuxieme}{Theorem}[section]
\newtheorem{troisieme}{Lemma}[section]
\newtheorem{quatrieme}{Proposition}[section]
\newtheorem{cinquieme}{Theorem}[section]

\begin{document}
\title{On the isoperimetric inequality of Coulhon and Saloff-Coste}
\author{Bruno Luiz Santos Correia}
\date{\today}
\begin{abstract}
We improve the isoperimetric inequality of Coulhon and Saloff-Coste following a method of mass transport proposed by Gromov.
\end{abstract}
\maketitle

\section{Introduction}

The isoperimetric inequality proved in 1993 by Coulhon and Saloff-Coste is related to the growth of finitely generated groups  \cite[Théorème1, page 295]{coulhon1993isoperimetrie}. 
Let $\Gamma$ be a finitely generated group and $S = \{ a_{1}^{\pm 1},... , a_{m}^{\pm 1} \} \subset  \Gamma$ a finite symmetric set of generators $\Gamma$. The length $|| g||$ of an element $g \in \Gamma$ is the minimal integer n such that $g$ can be written as $g=a_{i_{1}}^{\epsilon_{1}}  \cdots a_{i_{n}}^{\epsilon_{n}}$ with $\epsilon_{1},...,\epsilon_{n} \in \{ -1,+1 \}$.



Let $ e \in \Gamma$ be the identity element and $r \in \mathbb{N}$, we denote by $B(e,r)  \coloneqq \{ g \in \Gamma : || g|| \leq r\} $  the ball of center $e$ and radius $r.$ The growth function of $\Gamma$ (relative to $S$) 
 is defined by  
$\gamma (r) \coloneqq  \mathrm{Card}(B(e,r))\text{ for }r \in \mathbb{N}$
and the inverse growth function is defined by 
 $\phi_{S}(v) \coloneqq \min \{r \in \mathbb{N} : \mathrm{Card}(B(e,r))>v \}    
   \text{ for }v \in \mathbb{N}$.

The Coulhon  Saloff-Coste's inequality
 \cite[Theorem 3.2, page 296]{pittet1999amenable} tells us that for an infinite finitely generated  group  $\Gamma$, with $S$ as above and  
 for all $D \subset \Gamma$ finite, 
$
\partial_{C}D\coloneqq 
 \{ x \in D : 
 \exists 
 s
  \in S  :
  x\cdot s \in \Gamma
  \backslash D \} 
  $   
    we have:
  \[
\begin{aligned}  
  \frac{\mathrm{Card} (\partial_{C} D)}{\mathrm{Card}(D)} \geq \frac{1}{4 \cdot \mathrm{Card}(S) \cdot \phi_{S} (2\cdot \mathrm{Card}(D))}.
\end{aligned}
\]

This result can be expressed in terms of a slightly different definition of the boundaries. Following  \cite[page 348]{gromov2007metric} we define the boundary of a
  finite subset $D \subset \Gamma$ by $\partial D \coloneqq 
 \{ a \in \Gamma : dist(a, D)=1 \}$.
It is  straightforward to show that  $|\partial D| \leq |S| \cdot |\partial_{C} D|$. Hence the Coulhon Saloff-Coste's inequality is a consequence of the following inequality:
 \[
\begin{aligned}  
  \frac{\mathrm{Card} (\partial D)}{\mathrm{Card}(D)} \geq \frac{1}{4 \cdot \phi_{S} (2\cdot \mathrm{Card}(D))}.
\end{aligned}
\]

The main part of this  article focus on  improving the last inequality by a factor 2, by following Gromov's idea based on a mass transport method
\cite[pages 346 - 348]{gromov2007metric}.  

Our main result is the following:
\begin{cinquieme}
Let $\Gamma$ be a  non trivial finitely  generated group with generator $S= S^{-1},$ $\mathrm{Card}(S) < \infty$.

For all finite non empty subset $D \subset \Gamma $ such that $\mathrm{Card}(D)< \frac{\mathrm{Card}(\Gamma)}{2}$, we have:

\[
\frac{\mathrm{Card}( \partial  D)}{\mathrm{Card}(D)}
 >
 \frac{1}{2\cdot \phi_{S}(2 \cdot \mathrm{Card}(D))} \: \: .
\]
\end{cinquieme}

We note that
if $\Gamma$ is infinite then the hypothesis  
$\mathrm{Card}(D)< \frac{\mathrm{Card}(\Gamma )}{2}$ is always verified  since $\mathrm{Card}(D)< \infty$.



The inequality  in the Theorem improves the lower bound
in 
 \cite[Theorem 3.2, page 296]{pittet1999amenable}
by a factor 2, furthermore the choice of  the definition of  the boundary
$\partial D =\{a \in \Gamma : dist (a,D)=1\}$ allows us to obtain  a strict inequality in the previous theorem. However our inequality is not optimal: as an example if we choose $\Gamma= \mathbb{Z}$ and $S={\pm 1}$,  the non strict inequality can be improved by a factor 4.

\section{Acknowledgements}
I want to thank  Christophe Pittet, my Master thesis director, for the discussions on the subject, his availability, his advice and his patience. 
I thank you to Aitor Perez Perez, for  discussions on the subject of  amenability.
I want to thank Grégoire Schneeberger for his carefully reading of the manuscript and his remarks. I want to thank Pierre de la Harpe, for his remarks and clarifications on the subject.

\section{ Mass transport according to Gromov in a finitely generated group}
 Let $D\subset\Gamma$ be a finite subset, 
we transport D or an element of  $\Gamma$ by left translation. The left translation by  $\gamma \in \Gamma,$  is the map  $\delta \mapsto \gamma\delta$ , 
for any element  $\delta \in D$. Here $\Gamma$ is equipped with the  metric associated to $S$, which is invariant by right multiplication, namely $dist_{\Gamma}(\delta,\gamma \cdot \delta) = ||\delta \cdot (\gamma \cdot  \delta)^{-1}||=||\gamma||$, therefore $\delta$ is moved by a distance $||\gamma||$ 
and as we prove in the  lemma below, the amount of mass which is transported out of D does not exceed  $||\gamma|| \cdot  \mathrm{Card}(\partial D)$.
In the following we consider $d \in \mathbb{R}_{+}^{*}$.


 \begin{defdeuxieme}
Let $\varphi_{d}$ be the smoothing of $\varphi$  by the smoothing kernel:
\begin{equation}
\text{That is } S(x,y)=
 \left\{
\begin{array}{r c l}
  & \frac{1}{\mathrm{Card}(B(e,d))}  \quad for \quad   dist(x,y) \leq d  \:,   \\
  &0  \qquad \qquad \quad \: for \quad  dist(x,y) > d      \: , 
\end{array}
\right. \text{ for all  }x,y \in  \Gamma.
\end{equation}
We define $\varphi_{d}$ as:
  \[ \varphi_{d}(y)=
  \sum_{x \in \Gamma}^{} S(x,y) \cdot \varphi (x), \text{ for all }
\varphi \text{ probability density function} .\]
\end{defdeuxieme}
We compute $\varphi_{d}$ according to its definition for
  $\varphi \coloneqq  \mathbf{1}_{D}$ the characteristic function of D:
\[
\begin {aligned}
\varphi_{d}(y)
&=
 \sum_{x \in \Gamma}^{} S(x,y) \cdot \mathbf{1}_{D}(x)
 &=
  \sum_{x \in D}^{}S(x,y)
&=
\sum_{x \in D\cap B(y,d)}^{} \frac{1}{\mathrm{Card}(B(e,d))}
 &=
 \frac{\mathrm{Card}( D \cap B(y,d) )}{\mathrm{Card}(B(e,d))} .
\end {aligned}
\]

\begin{deuxieme}
  $($\cite[page 343]{gromov2007metric}$ )$
  Let $\Gamma$ a group finitely generated by $S=S^{-1}, \mathrm{Card}(S)<\infty$.

 For all finite subsets $D \subset \Gamma$, 
 for all $ d \in \mathbb{N}$ and 
 for $\varphi \coloneqq  \mathbf{1}_{D}$
 we have:
\[
\begin{aligned}
&i)  \sum_{y\in D}^{}|\varphi(y) - \varphi_{d}(y)  | 
=
\frac{1}{\mathrm{Card}(B(e,d))}\sum_{x\in B(e,d)}^{} \mathrm{Card}( xD\setminus D) ,\\
& ii) \sum_{y \in  D}^{} \mathrm{Card}(B(y,d) \setminus D )  = \sum_{x \in B(e,d) }\mathrm{Card}(xD \setminus D).
\end{aligned}
\]
\end{deuxieme}
\begin{proof}

We calculate the following variation:
\begin{align}
\sum_{y \in D}^{} |\varphi(y)-\varphi_{d}(y)|
&= 
\sum_{y \in D}^{} \frac{|\mathrm{Card}(B(e,d))-\mathrm{Card}(D\cap B(y,d))|}{\mathrm{Card}(B(e,d))}  \nonumber \\
&=
\sum_{y \in D}^{} 
\frac{\mathrm{Card}(B(y,d)     \setminus  [ D\cap B(y,d)])}{\mathrm{Card}(B(e,d))} \nonumber \\
&= 
\sum_{y \in D}^{} \frac{\mathrm{Card} (B(y,d) \setminus D)}{\mathrm{Card}(B(e,d))}  \nonumber \\
&=
\frac{1}{\mathrm{Card}(B(e,d))}\sum_{y \in D}^{}\mathrm{Card}(B(y,d)\smallsetminus D)  .
\end{align}

Due to the formula (2) it is sufficient to prove $ii)$:

\[
\sum_{y \in D}^{} \mathrm{Card}(B(y,d) \smallsetminus D) 
 =
\sum_{x\in B(e,d)}^{} \mathrm{Card}( xD\smallsetminus D) \: \: .
\]

We notice that for all  $y\in D$ , 

\begin{equation}
\begin{aligned}
\mathrm{Card}(B(y,d) \smallsetminus D) 
&=
 \mathrm{Card}( \{ z \in \Gamma : dist(y,z) \leq d   \} \smallsetminus D )\\
 &=
 \mathrm{Card}( \{ z \in  \Gamma \setminus D : dist(y,z) \leq d\}  )\\
 &= 
 \sum_{x \in B(e,d)}^{} \mathbf{1}_{\Gamma \smallsetminus D}(x \cdot y)\\
 &=
 \sum_{x \in B(e,d)}^{} \mathbf{1}_{D^{c}}(x \cdot y)  \: \: .
\end{aligned}
 \end{equation}
 
 Furthermore for  $x\in \Gamma $ we get: 

 \begin{equation}
 \begin{aligned}
 \mathrm{Card}(xD \setminus D)
&=
  \mathrm{Card}( \{ x \cdot y : y \in D \} \setminus \ D) \\
&=
\mathrm{Card}(  \{ y \in D : x \cdot y \in \Gamma \setminus D \}  )\\
&=
\sum_{y \in D}^{} \mathbf{1}_{\Gamma \setminus D}(x \cdot y)\\
&=
\sum_{y \in D}^{} \mathbf{1}_{ D^{c}}(x \cdot y) \: \: .
\end{aligned} 
\end{equation}

Thus:

\[
\begin{aligned}
\sum_{y \in D}^{} \mathrm{Card}(B(y,d) \setminus D)  
\overset{(3)}{=}
\sum_{y \in D}^{} \sum_{x\in B(e,d)}^{} \mathbf{1}_{D^{c}}(x \cdot y)
=
\sum_{x \in B(e,d) }^{} \sum_{y\in D}^{} \mathbf{1}_{D^{c}}(x \cdot y)
\overset{(4)}{=}
\sum_{x\in B(e,d)}^{} \mathrm{Card}( xD\setminus D).
\end{aligned}
\]
\end{proof}

 

 
\begin{sixieme}
Le $\Gamma$ be a finitely generated group and $S$ a finite symetric set as above. Let $D \subset \Gamma$ be non empty and let $\Gamma$ be non trivial. 

Assume $ \mathrm{Card}(D)<
 \frac{ \mathrm{Card}(\Gamma) }{2}$. 
 Let $ d\in \mathbb{N} $ be minimal such that 
 $ \mathrm{Card}(B(e,d))   >2 \cdot \mathrm{Card}(D)$.
 
 Then there exists $x \in B(e,d)$ such that  
 $\mathrm{Card}(xD \setminus D)  > \frac {\mathrm{Card}(D)}{2}$.
\end{sixieme}
\begin{proof}

 Let $\varphi = \mathbf{1}_{D}$. We have:
\begin{equation}
\sum_{y\in D}^{}|  \underbrace{ \varphi(y)}_{=1} - 
\underbrace{ \varphi_{d}(y)}_{<\frac{1}{2}} | > \frac{1}{2}\cdot \mathrm{Card}(D).
\end{equation}

Let $f: \Gamma \rightarrow \mathbb{R}_{+}, \: f(x)= \mathrm{Card}(xD \setminus D).$ We have:

\[
\begin{aligned}
\frac{1}{\mathrm{Card}(B(e,d))} \cdot \sum_{y \in B(e,d)} f(y) 
\overset{ii) \text{Lemma above}}{=}
\frac{1}{\mathrm{Card}(B(e,d))} \cdot \sum_{y \in D} \mathrm{Card}(B(y,d) \setminus D).
\end{aligned}
\]

Since $\mu (y)=  \left\{
\begin{array}{r c l}
  & \frac{1}{\mathrm{Card}(B(e,d))}  \quad if \quad   y \in B(e,d)  \:,   \\
  &0   \quad  if  \quad y \notin B(e,d),     
\end{array}
\right.$
 defines a measure of probability on $\Gamma$ and since $f$ is positive we have:\\
 \[
 \sum_{x \in \Gamma} f(x)\cdot \mu(x) 
 =
 \sum_{x \in B(e,d)}f(x)\cdot \mu(x)
 \leq
 \max_{x\in B(e,d)} f(x)\cdot \sum_{x\in B(e,d)} \mu (x)
 =
 \max_{x \in B(e,d)} f(x).
 \]
 
 Which implies:
\begin{equation}
\begin{aligned}
\exists x \in B(e,d): f(x) \geq \frac{1}{\mathrm{Card}(B(e,d))}\cdot \sum_{y \in D}^{}\mathrm{Card}(B(y,d) \setminus D) .
\end{aligned}
\end{equation}

 By (5), (6) and $(i)$ Lemma3.1 we obtain that  :
 \begin{equation}
 \mathrm{Card}(xD \setminus D)
  >\frac{1}{2}\cdot \mathrm{Card}(D)  \text{, for a } x \in B(e,d).
 \end{equation}
 This ends the proof of the Proposition.
\end{proof}

Now we want to show that:
\[
d \cdot \mathrm{Card}(\partial D) 
 \geq \mathrm{Card}(xD \setminus D), \text{ for }  ||x|| \leq d.
\]
\begin{quatrieme}

Let  $D \subset \Gamma$ be  a  finite subset.

Let $\: \gamma_{0} \in \Gamma \: such \:that \:
||\gamma_{0} ||_{S} \leq d.$

 Let $ \partial D=\{\gamma \in \Gamma: d(\gamma, D)=1\}.$
 
Then  $ \mathrm{Card}(\gamma_{0} D \setminus D)
  \leq d \cdot \mathrm{Card}( \partial D ).$
  \end{quatrieme}
  \begin{proof}
First chosse $s_{1},...,s_{k} \in S$ such that 
$\gamma_{0}
= 
s_{k}
 \cdots
  s_{1}$ with
 $k=|| \gamma_{0}||_{S}$. 
  
 Let $ x \in \gamma_{0}D \setminus D$.
 Then  $x=\gamma_{0}\omega_{x}$ with 
 $\omega_{x} \in D$, therefore $\gamma_{0}^{-1} \cdot x= \omega_{x} \in D.$
 
 
We note $M_{x}= {\underset{1 \leq n \leq k}\max } \{ n: s_{n} \cdot s_{n-1}\cdot \cdot \cdot s_{1}\cdot \omega_{x} \in \partial D \}$.

We define  
\[  
  \begin{aligned}
f:\gamma_{0} D \setminus D & \longrightarrow \partial D\\
 x & \mapsto s_{M_{x}} \cdot s_{M_{x}-1} \cdot \cdot \cdot s_{1}\cdot \omega_{x}, \\
 \end{aligned}
 \]
 where geometrically $f(x)$ is the first point of  $\partial D$
that intersect the geodesic path from $x$ to $\omega_{x}$ defined by the choice of $s_{1},...,s_{k}$. (Remenber that our word metric is right invariant). 
We illustrate geometrically this definition  in Figure 1 and Figure 2.

\begin{figure}[h]
\begin{center}

\includegraphics[trim={8.6cm 5.8cm 10.5cm 4.6cm},clip=true,scale=0.46]{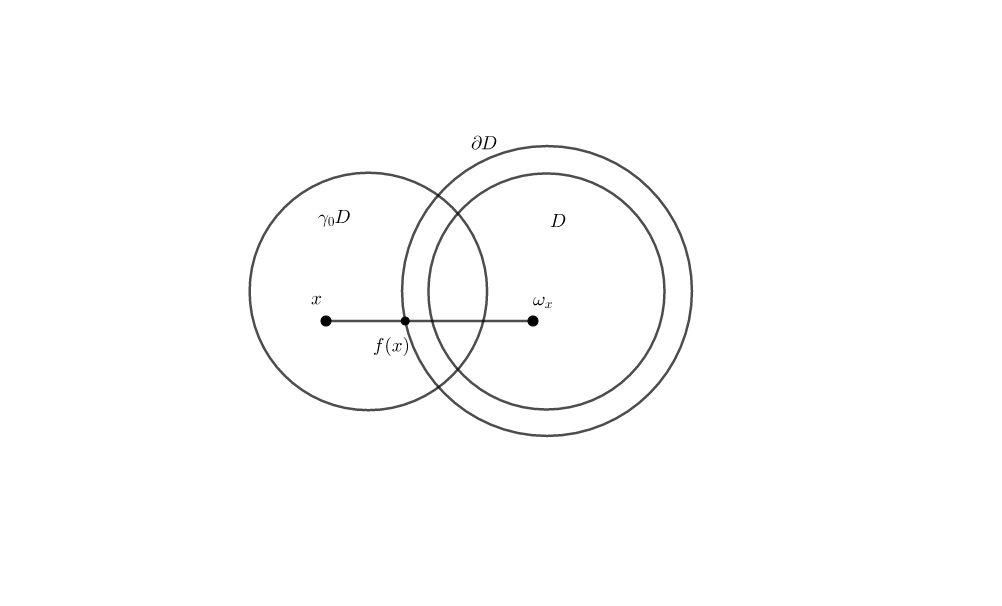}
\end{center}
\caption{Geometric representation of the application $f$.}
\end{figure}

\begin{figure}[h]
\begin{center}
\includegraphics[trim={6cm 8cm 9cm 4.8cm},clip=true,scale=0.46]{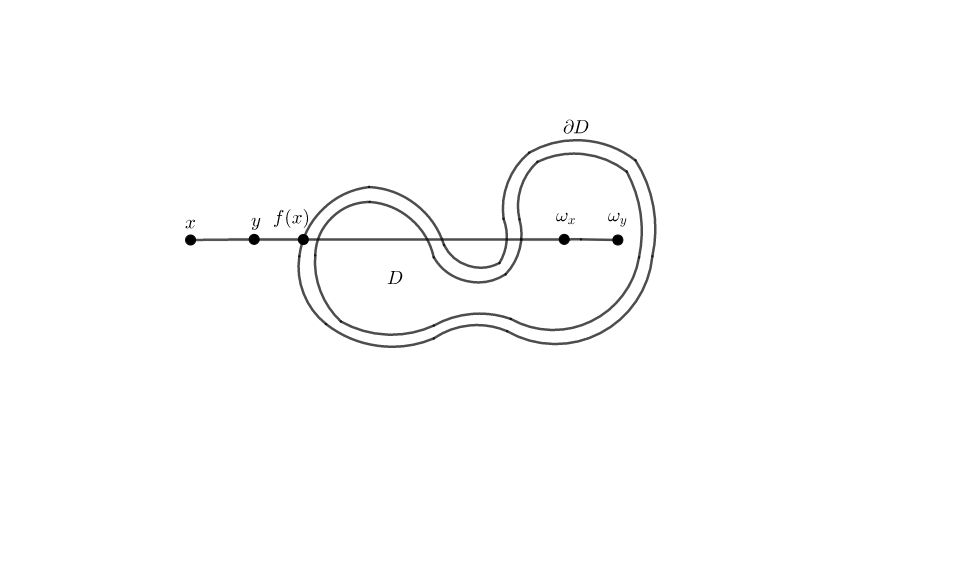}
\end{center}
\caption{Geometric representation allowing to see  $M_{x}$.}
\end{figure}

Let us first demonstrate the following lemma.

 \begin{troisieme}
 
 Let $\: z\in \partial D$. 
 
Then $\: \mathrm{Card}( f^{-1}(z)) \leq d.$
\end{troisieme} 
 \begin{proof}
 
If $f^{-1}(z)$ is non empty, we fix  $x \in f^{-1}(z)$  such  that  $ M_{x}$ is maximal among the points of $f^{-1}(z) .$ 
 Let $y \in f^{-1}(z).$ 

We compare the writing of  $f(x)$ with the one of $f(y)$

$x=\gamma_{0}\cdot  \omega_{x}$

$y=\gamma_{0}\cdot  \omega_{y},$ with $ \omega_{x}, \omega_{y} \in D$

\[
\begin {aligned}
&f(x)=z=f(y) \Leftrightarrow
s_{M_{x}} \cdot s_{M_{x}-1} \cdots s_{1} \omega_{x}=z=s_{M_{y}} \cdot s_{M{y}-1}\cdots s_{1}\cdot \omega_{y}\\
& \text{ with }  M_{x} \geq M_{y}\\
&  \overbrace{s_{M_{x}} \cdot s_{M_{x}-1} \cdots}^{h} \underbrace{s_{M_{y}} \cdots s_{1}}_{g}\cdot \omega_{x}=  
\underbrace{s_{M_{y}}\cdots s_{1}}_{g}   \cdot \omega_{y}
\Leftrightarrow
hg\omega_{x}=g\omega_{y}\\
& \Leftrightarrow g^{-1}hg\omega_{x}=\omega_{y}.
\end{aligned}
\]

Thus $y=\gamma_{0}\omega_{y}=\gamma_{0}g^{-1}hg\omega_{x}$
and $g^{-1}hg$ is completely determined by the value of $M_{y}$. We have $1 \leq M_{y} \leq M_{x} \leq k \leq d$. Thus we have $d$ possibilities for  $y$.

This completes the proof of the lemma.
 \end{proof}
 
We consider the application  $f:\gamma_{0} D \backslash D \rightarrow  \partial D$. According to the previous result each point of $\partial D$ has at most $d$ preimages. Therefore 
\begin{equation}
\begin{aligned}
 \mathrm{Card}(\gamma_{0} D \backslash D ) \leq d \cdot \mathrm{Card}(\partial D) .
\end{aligned}
 \end{equation}


This ends  the  proof of the proposition.
\end{proof}

As mentioned in the introduction the following theorem is our main result.

\begin{deuxieme}
Let $\Gamma$ be a non trivial group finitely generated by $S= S^{-1},$ $\mathrm{Card}(S) < \infty$.

For all finite non empty subset $D \subset \Gamma $ such that $\mathrm{Card}(D)< \frac{\mathrm{Card}(\Gamma)}{2}$, we have:

\[
\frac{\mathrm{Card}( \partial  D)}{\mathrm{Card}(D)}
 >
 \frac{1}{2\cdot \phi_{S}(2 \cdot \mathrm{Card}(D))} \: \: .
\]

\end{deuxieme}
\begin{proof}

By definition  $\phi_{S}(v)=min\{r: \mathrm{Card}(B(e,r))>v\}$.

We choose  $d$  minimum such that   $\mathrm{Card}(B(e,d)) > 2\cdot \mathrm{Card}(D)$, which means $\phi_{S}(2 \cdot \mathrm{Card}(D))=d $.

Furthermore, applying Proposition 3.1 and Proposition (Half mass transportation), we obtain: $\: \: d \cdot \mathrm{Card}( \partial D) > \frac{\mathrm{Card}(D)}{2}$ therefore

\[
\frac{\mathrm{Card}( \partial D)}{\mathrm{Card}(D)} > \frac{1}{2 \cdot d} = \frac{1}{2\cdot \phi_{S}(2\cdot \mathrm{Card}(D))}.
\]

This finishes the  proof of the theorem.
\end{proof}

\bibliography{mabiblio}{}

\begin{thebibliography}{1}

\bibitem{coulhon1993isoperimetrie}
Thierry Coulhon and Laurent Saloff-Coste.
\newblock Isop{\'e}rim{\'e}trie pour les groupes et les vari{\'e}t{\'e}s.
\newblock {\em Revista Matem{\'a}tica Iberoamericana}, 9(2):293--314, 1993.

\bibitem{gromov2007metric}
Mikhail Gromov.
\newblock {\em Metric structures for Riemannian and non-Riemannian spaces}.
\newblock Springer Science \& Business Media, 2007.

\bibitem{pittet1999amenable}
Christophe Pittet and Laurent Saloff-Coste.
\newblock Amenable groups, isoperimetric profiles and random walks.
\newblock {\em Geometric group theory down under (Canberra, 1996)}, pages
  293--316, 1999.

\end{thebibliography}
\bibliographystyle{plain}

\end{document}